\documentclass{article}
\usepackage[utf8]{inputenc}

\usepackage{amsmath}
\usepackage{amsthm}
\usepackage{amssymb}

\newtheorem{thm}{Theorem}
\newtheorem{lem}[thm]{Lemma}
\newtheorem{defn}[thm]{Definition}

\newtheorem{prop}[thm]{Proposition}

\newcommand{\ff}[1]{\mathbb{F}_{#1}}

\newcommand{\GL}[2]{\textbf{GL}_{#1}(#2)}

\newcommand{\G}{\mathcal{G}}

\usepackage{amssymb, amsmath}
\begin{document}

\title{A note on the automorphism group of Schubert varieties}
\author{Fernando L. ~Pi\~nero}

\maketitle

\begin{abstract}

In \cite{GK13}, the authors determined that the automorphisms of a Schubert divisor are those automorphisms which fix a particular subspace. In this work we extend those results to all Schubert varieties.  We study the Schubert conditions which define a Schubert variety and the action upon these conditions by the automorphism group of the Grassmannian variety. We concluthe that the automorphisms of the Grassmannian which map a Schubert variety to itself if and only if it fixes the subspaces which do not give redundant conditions used to define the Schubert variety. 

\end{abstract}



\section{Introduction}
In this article we let $V = \ff{q}^m$ with its usual $\ff{q}$--linear vector space structure. We denote by $[a] := \{0,1,2,\ldots, a\}$. We also consider a subset $\alpha \subseteq [m]$ as an ordered tuple. That is $\alpha = (a_1 < a_2 < \ldots < a_{\ell})$. 

\begin{defn}  A \emph{flag of $V$} is a sequence of nested subspaces $$\mathcal{A} := A_1 \subsetneq A_2 \subsetneq A_3 \subsetneq \cdots \subsetneq A_\ell \subseteq V.$$

For $\alpha = (a_1, a_2, \ldots, a_\ell)$ be a subset of $[m]$ we denote the flag $$\mathcal{A} := A_1 \subsetneq A_2 \subsetneq A_3 \subsetneq \cdots \subsetneq A_\ell \subseteq V$$ as an \emph{$\alpha$--flag} if $\dim A_i = a_i$.

\end{defn}

Note that in the case of an $\alpha$--flag, there exists a basis $\mathbf{a}_1, \mathbf{a}_2, \ldots, \mathbf{a}_{m}$ such that $A_i$ is spanned by $\{  \mathbf{a}_j \ | \ 1 \leq j \leq a_i \}.$

\begin{defn} \emph{The $\ell$--Grassmannian of $\ff{q}^m$} is the set of all subspaces of $\ff{q}^m$ whose dimension is $\ell$ that is: $$\G_{\ell, m} := \{ W \leq \ff{q}^m \ | \ \dim W = \ell \}.$$
\end{defn}

\begin{lem}
 A matrix $M \in \GL{m}{\ff{q}}$ acts on a row vector $\mathbf{x} \in \ff{q}^m$ by mapping $\mathbf{x}$ to the vector $\mathbf{x} M$. This action is extended to a subspace $W \leq \ff{q}^m$ as follows: If $W = \langle \mathbf{w}_1, \mathbf{w}_2, \ldots, \mathbf{w}_r \rangle$, then $M$ maps $W$ to $M(W) :=  \langle \mathbf{w}_1M, \mathbf{w}_2M, \ldots, \mathbf{w}_rM \rangle$.
\end{lem}

\begin{lem}
 Let $\theta$ be a field automorphism of $\ff{q}$. This automorphism $\theta$ acts on the space $\ff{q}^m$ by mapping the vector $(x_1, x_2, \ldots, x_m) = \mathbf{x} \in \ff{q}^m$  to the vector  $\mathbf{x}^\theta := (\theta(x_1), \theta(x_2), \ldots,\theta(x_m)) \in \ff{q}^m$. This is extended to a subspace $W \leq \ff{q}^m$ as follows: If  $W = \langle \mathbf{w}_1, \mathbf{w}_2, \ldots, \mathbf{w}_r \rangle$, then the automorphism $\theta$ maps $W$ to $ W^\theta := \langle \mathbf{w}_1^\theta, \mathbf{w}_2^\theta, \ldots, \mathbf{w}_r^\theta  \rangle$.
\end{lem}

\begin{defn}
For $\alpha \subseteq [m]$, we denote the set $m - \alpha:= \{ m - a_i \ | \ a_i \in \alpha\}$.
\end{defn}

W.L. Chow proved the following:
\begin{prop}\cite[Chow]{C49}

Let $1 < \ell < m-1$. The permutations of $\G_{\ell, m}$ which map lines to lines is given by the group ${ \displaystyle \operatorname {\Gamma L}(\ff{q}) }$. That is, these permutations are given by compositions of the following permutations:
\begin{itemize}
\item The permutation $\sigma_M$ where $\sigma_M(W) = W.M$  for $M \in \GL{m}{\ff{q}}$.
\item The permutation $\sigma_\theta$ where $\sigma_\theta(W) = W^\theta$ for $\theta$ a field automorphism of $\ff{q}$, 
\item  If $\ell = m - \ell$, the permutation $\sigma_\perp$ where $\sigma_\perp(W) = W^{\perp}$ where $W^\perp$ is the orthogonal complement of $W$.
\end{itemize}
\end{prop}

With an orthogonal basis for $V$, the permutation $\sigma_\perp$ is also given by the Hodge star operator on $\bigwedge^\ell V$. Although the relation between $\bigwedge^\ell V$ and $\G_{\ell, m}$ is well known, for our purposes, we need only to consider the permutations of $\G_{\ell, m}$ onto itself given by the elements of ${ \displaystyle \operatorname {\Gamma L}(\ff{q}) }$. Note that $\sigma_\perp$ maps $\G_{\ell, m}$ onto $\G_{m-\ell, m}$. As such we will also consider ${ \displaystyle \operatorname {\Gamma L}(\ff{q}) }$ acting on $\bigcup_{i=1}^{m-1}\G_{i, m}$. This action is extended to flags as follows. 
\begin{defn}
 
Let $M \in \GL{m}{\ff{q}}$ and $\theta$ be a field automorphism.  Suppose $\alpha = (a_1, a_2, \ldots, a_\ell)$. Let $\mathcal{A} :=  A_1 \subsetneq A_2 \subsetneq A_3 \subsetneq \cdots \subsetneq A_{\ell}$ be an $\alpha$--flag.  Then We define:

\begin{itemize}
\item The linear transformation $M$ maps the $\alpha$-- flag $\mathcal{A}$ to the $\alpha$--flag: $$M(\mathcal{A}) := M(A_1) \subsetneq M(A_2) \subsetneq \cdots \subsetneq M(A_\ell).$$
\item The field automorphism $\theta$ maps the $\alpha$--flag $\mathcal{A}$ to the $\alpha$--flag $$\mathcal{A}^\theta := A_1^\theta \subsetneq A_2^\theta \subsetneq \cdots \subsetneq A_\ell^\theta.$$
\item The orthogonal complement, $\perp$ maps the $\alpha$--flag $\mathcal{A}$ to the $m-\alpha$--flag $$\mathcal{A}^{\perp} :=  A_{\ell}^\perp \subsetneq \cdots \subsetneq A_1^\perp.$$

\end{itemize}
\end{defn}

\section{Schubert Varieties}

Schubert varieties are special subvarieties of $\G_{\ell, m}$. By considering Schubert subvatieties, one can answer many geometrical questions about projective spaces in general and study the Grassmannian as well. The classical reference to Schubert varieties is \cite{KL72}. 

\begin{defn}
 Let $$ \alpha = (a_1 < a_2 < \cdots < a_\ell) \subseteq [m].$$  Let $\mathcal{A}$ be an $\alpha$--flag. The Schubert variety is defined as $$\Omega^{\mathcal{A}}_\alpha :=  \{ W \in \G_{\ell,m} \ | \ \dim(W \cap A_i) \geq i \}.$$ 
\end{defn}

We have included the $\alpha$--flag $\mathcal{A}$ in the notation for the Schubert variety $\Omega_\alpha^{\mathcal{A}}$ because we shall consider what happens to the Schubert varieties when the flag is changed. For any two  $\alpha$--flags, $\mathcal{A}$ and $\mathcal{B}$, the varieties $\Omega_\alpha^{\mathcal{A}}$ and $\Omega_\alpha^{\mathcal{B}}$ are isomorphic. However, the choice of flag may change the Schubert variety.

Some of the Schubert conditions $\dim (W \cap A_i) \geq i$ may be redundant.  Suppose $\alpha \subseteq [m]$ has two consecutive elements, say $a_i = a_{i -1}+1$. Each $W \in \Omega_{\alpha}^{\mathcal{A}}$ satisfies $\dim(W \cap A_i) \geq i$. As $\dim A_i = a_i$ and $\dim A_{i-1} = \dim A_{i} -1$ the inequality $\dim(W \cap A_{i}) \geq i$ implies $\dim(W \cap A_{i-1}) \geq i-1$. Therefore the condition $\dim(W \cap A_{i-1}) \geq i-1$ is redundant.  This motivates the following definition.

\begin{defn}
Let $\alpha \subseteq [m]$. We define the \emph{nonconsecutve subset of $\alpha$} as $$\alpha_{nc} := \{ a_i \ | \ a_i +1 \not\in \alpha  \}.$$

\end{defn}

The previous discussion implies the following.

\begin{lem}
Let $\alpha = (a_1, a_2, \ldots, a_{\ell} ) \subseteq [m]$. Suppose $$\mathcal{A} := A_1 \subsetneq A_2 \subsetneq A_3 \subsetneq \cdots \subsetneq A_\ell \subseteq V$$ and $$\mathcal{B} := B_1 \subsetneq B_2 \subsetneq B_3 \subsetneq \cdots \subsetneq B_\ell \subseteq V$$ are two $\alpha$--flags.

If $A_i = B_i  \ \forall i \in \alpha_{nc}$, then $$\Omega_{\alpha}^{\mathcal{A}} = \Omega_{\alpha}^{\mathcal{B}}.$$
\end{lem}
\begin{proof}
As we have discussed, the conditions given by $A_i$ and $B_i$ where $\dim A_i = \dim B_i \in \alpha_{nc}$ imply the remaning conditions. By hypothesis, $A_i = B_i$ whenever $\dim A_i = \dim B_i \in \alpha_{nc}$. Equality follows. \end{proof}

Laksov and Kleiman  \cite{KL72} proved that two Schubert varieties are isomorphic if and only if they have the same dimension sequence. Therefore we have stated that propositon as follows.

\begin{prop}[\cite{KL72}, Proposition 4]

$$\Omega_{\alpha}^{\mathcal{A}} = \Omega_{\beta}^{\mathcal{B}} \makebox{ implies }\alpha = \beta.$$
\label{prop:KL}
\end{prop}
 Now we have the following theorem.

\begin{thm}
\label{thm:SchubertEqualityTheorem}
Let $ \alpha = (a_1, a_2, \ldots, a_{\ell} )\subseteq [m]$.  Let $\mathcal{A}$ and  $\mathcal{B}$ be two $\alpha$--flags. Then $$\Omega_{\alpha}^{\mathcal{A}} = \Omega_{\alpha}^{\mathcal{B}} \makebox{ if and only if } A_{i} = B_{i}, \forall a_i  \in \alpha_{nc} .$$  

\end{thm}
\begin{proof}

From the previous discussion, the veracity of the only if direction is clear. 
 
Let $a_s$ be the largest element in $\alpha_{nc}$ such that $A_s \neq B_s$. Let $a_{r_1}$ be the next smallest index in $\alpha_{nc}$ and let Let $a_{r_2}$ be the next largest index in $\alpha_{nc}$. The choice of $s$ implies $A_{r} = B_{r}$ for any index in $\alpha_{nc}$ greater than $s$.

As $\mathcal{A}$ and $\mathcal{B}$ are $\alpha$--flags, there exists $\mathbf{a}_1, \mathbf{a}_2, \ldots, \mathbf{a}_{m}$ such that $A_i$ is spanned by $\{  \mathbf{a}_j \ | \ 1 \leq j \leq a_i \}$, and there exists $\mathbf{b}_1, \mathbf{b}_2, \ldots, \mathbf{b}_{m}$ such that $B_i$ is spanned by $\{  \mathbf{b}_j \ | \ 1 \leq j \leq a_i \}$.

If $a_s = a_\ell$ is the largest element, there exists $\mathbf{x} \in A_\ell \setminus B_\ell$. The vector space $W$ spanned by $\mathbf{a}_1, \mathbf{a}_2, \ldots, \mathbf{a}_{\ell-1} $ and $\mathbf{x}$ is in  $\Omega_{\alpha}^{\mathcal{A}}$ but not in $\Omega_{\alpha}^{\mathcal{B}}$. Thus $\Omega_{\alpha}^{\mathcal{A}} \neq \Omega_{\alpha}^{\mathcal{B}}$. 

If $a_s$ is not the largest element in $\alpha_{nc}$. Note that $A_s \neq B_s$ but $$A_{r_1} = B_{r_1}  \subseteq A_s,B_s \subseteq A_{r_2} = B_{r_2}.$$ Let $\mathbf{x} \in A_s \setminus B_s$. In this case consider the vector space $W$ spanned by the set $ \{ \mathbf{a}_{a_u}  \ | \ 1 \leq u \leq \ell,  u \neq s \}  \cup \{ \mathbf{x}\}$. In this case $\dim W \cap A_{u} = u$ for each $u \in \alpha_{nc}$, but $\dim W \cap B_s = s-1$. Therefore $W \in \Omega_{\alpha}^{\mathcal{A}}$ but not in $\Omega_{\alpha}^{\mathcal{B}}$. \end{proof}

Now we aim find the automorphism group of $\Omega_\alpha^{\mathcal{A}}.$

\begin{lem}
\label{lem:AutSchubert}
Let $\alpha = (a_1, a_2, \ldots, a_{\ell} ) \subseteq [m]$. Suppose $$\mathcal{A} := A_1 \subsetneq A_2 \subsetneq A_3 \subsetneq \cdots \subsetneq A_\ell \subseteq V$$ is an $\alpha$--flag. Let $\tau \in Aut(\G_{\ell, m}))$.  Suppose $\tau$ preserves the dimension of any linear subspace of $V$. Then  $\tau(\Omega_{\alpha}^\mathcal{A}) = \Omega_\alpha^{\tau(\mathcal{A})}$.

\end{lem}
\begin{proof} 

The Schubert variety $\Omega_{\alpha}^\mathcal{A}$ is defined by $$\{ W \in \G_{\ell, m} \ | \ \dim W \cap A_i \geq i\}.$$ The automorphism $\tau \in Aut(\G_{\ell, m})$ maps $\Omega_\alpha^{\mathcal{A}}$ to $$\tau(\Omega_{\alpha}^\mathcal{A })  = \{ \tau(W) \in \G_{\ell, m} \ | \ \dim \tau(W \cap A_i)\geq i\}.$$ In this case, $\tau(W \cap A_i) = \tau(W) \cap \tau(A_i)$. As $\tau$ is a permutation of the Grassmannian, we change the indexing variable to $\tau(W) = U$. Now the Schubert variety has the form: $$\tau(\Omega_{\alpha}^\mathcal{A })  = \{ U \in \G_{\ell, m} \ | \ \dim U \cap \tau(A_i)\geq i\}.$$ The right hand side is clearly $\Omega_\alpha^{\tau(\mathcal{A})}$ and equality follows.  \end{proof}

\begin{thm}
\label{thm:Variant}
Let $\alpha = (a_1, a_2, \ldots, a_{\ell} ) \subseteq [m]$. Suppose $$\mathcal{A} := A_1 \subsetneq A_2 \subsetneq A_3 \subsetneq \cdots \subsetneq A_\ell \subseteq V$$ is an $\alpha$--flag. Let $\tau \in Aut(\G_{\ell, m}))$.  Suppose $\tau$ preserves the dimension of any linear subspace of $V$. Then  $\tau \in Aut(\Omega_{\alpha}^{\mathcal{A}})$ if and only if $\tau(A_i) = A_i \forall a_i \in \alpha_{nc}$.

\end{thm}
\begin{proof} 

Lemma \ref{lem:AutSchubert} implies $\tau \in Aut(\G_{\ell, m})$ maps $\Omega_\alpha^{\mathcal{A}}$ to $\Omega_\alpha^{\tau(\mathcal{A})}.$ Theorem \ref{thm:SchubertEqualityTheorem} implies $\Omega_\alpha^{\mathcal{A}} = \Omega_\alpha^{\tau(\mathcal{A})}$ if and only if $\tau(A_i) = A_i \forall a_i \in \alpha_{nc}$.  \end{proof}


When $\ell \neq m- \ell$ the only line preserving bijections are those which preserve the dimension. On the remainder of the article, we shall assume $\ell = m - \ell$. Now we shall determine what happens when $\tau \in Aut(\G_{\ell, m})$ is a contravariant mapping. That is when $\dim \tau(W) = m - \dim W.$ In this case we do know $\tau$ maps the $\alpha$--flag $\mathcal{A}$ to the $m-\alpha$--flag $\tau(\mathcal{A})$, but it may be possible that due to the change of dimensions, that the Schubert conditions $\dim A_i \cap W \geq i$ might change. Now we study how $\tau$ might change the conditions $\dim A_i \cap W \geq i$. In this case we shall make use of the notion of a complete flag.

\begin{defn}
A \emph{complete flag} is a $[m]$--flag. That is, it is a flag which contains a subspace of each dimension. That is a complete flag is a sequence of nested subspaces $\mathcal{C} =  C_0 = \{ 0\} \subsetneq C_1 \subsetneq C_1 \subsetneq \cdots \subsetneq C_m = \ff{q}^m$ where $\dim C_i = i$.

If a complete flag $\mathcal{C}$ contains the subspaces $A_i$ where  $$\mathcal{A} := A_1 \subsetneq A_2 \subsetneq A_3 \subsetneq \cdots \subsetneq A_\ell \subseteq V$$ then $\mathcal{A}$ is known as a \emph{subflag} of $\mathcal{C}$.
\end{defn}

\begin{lem}
Let $\tau \in Aut(\G_{\ell, m})$ be a contravariant mapping. Let $\tau(\Omega_\alpha^{\mathcal{A}})$ be the image of $\Omega_\alpha^{\mathcal{A}}$. Then $\tau(\Omega_\alpha^{\mathcal{A}}) =  \Omega_\beta^{\tau(\mathcal{A})}$  where $\beta = \{  m+1 - j | \  j \not\in \alpha\}.$ 
\label{lem:dualSchubertaction}
\end{lem}
\begin{proof}

Let $\mathcal{A} = A_1 \subsetneq A_2 \subsetneq A_3 \subsetneq \cdots \subsetneq A_\ell \subseteq V$ be an $\alpha$--flag. Suppose $\mathcal{C} = C_0 = \subsetneq C_1 \subsetneq C_1 \subsetneq \cdots \subsetneq C_m$ is a complete flag with $\mathcal{A}$ as a subflag.

The Schubert conditions $\dim A_i \cap W \geq i$ can be extended to the subspaces of $\mathcal{C}$ Simply note that for $A_i \subseteq C_s \subsetneq A_{i+1}$ the condition $\dim C_s \cap W \geq i$ holds. Now we have the following Schubert conditions on the complete flag $\mathcal{C}$. $$ \dim C_s \cap W \geq i, \makebox{ for } a_i \leq s < a_{i+1}. $$ 

Thus given $\alpha$, an $\alpha$--flag $\mathcal{A}$ and a complete flag $\mathcal{C}$ containing $\mathcal{A}$ we may rewrite the conditions as follows: Let $w_0, w_1, w_2, \ldots, w_m$ be a sequence of integers such that $w_s = i$ for $a_i \leq s <a_{i+1}$. Then 
$$ \dim C_s \cap W \geq w_s.$$ Note that $w_i$ increases by $1$ only on the positions corresponding to $\alpha$. That is $\alpha = \{ s \ | n_s = n_{s-1}-1 \}$.

Now we shall apply $\tau$ to $ \dim C_s \cap W \geq w_s$. The Schubert conditions become $$ \dim \tau(C_s \cap W) \leq m - w_s.$$ As $\tau(C_s \cap W) $ is the vector space spanned by $\tau(C_s)$ and $\tau(W)$, we have the conditions
$$ \dim \tau(C_s ) + \tau(W) \leq m - w_s.$$  This is equivalent to $$ \dim \tau(C_s ) + \dim \tau(W) -  \dim \tau(C_s) \cap \tau(W) \leq m - w_s.$$ 

 In order to simply our notation we shall set $r = m+1-s$, $D_r = \tau(C_{m+1-s})$, $\tau(W) = U$, and $u_{r} = m - w_{s}$. Note that $\tau$ maps $\G_{\ell, m}$ to itself so $U$ also represents any element of the Grassmannian. The Schubert conditions become

$$ \dim D_r + \dim U -  \dim D_r \cap U \leq u_r.$$ 
As $\dim D_r = r$, $\dim U = \ell$ we rearrange the terms and obtain: 
$$ \dim D_r \cap U \geq r + \ell -  u_r.$$

Let $n_r = r + \ell -  u_r$. Now we determine $\beta = \{ j \in [m] \ | \ n_{j+1} = n_j +1 \}$.  Recall that these are the entries where $\dim D_{j} \cap U > \dim D_{j-1} \cap U$.  In this case there are some stringent conditions on$\beta$ from the equality $\tau(\Omega_\alpha^{\mathcal{A}}) =  \Omega_\beta^{\tau(\mathcal{A})}$ and Proposition \ref{prop:KL}.

Suppose $n_r = n_{r+1}$.  In this case $ r + \ell -  u_r =  r+1 + \ell -  u_{r+1}$. From the definition of $r$ and $u_r$ we have that $m+1 -s + \ell - (m - w_{m+1-s}) = m-s + \ell - (m - w_{m-s}).$ Therefore increases in $\dim D_r \cap U$ do not occur for $w_{m-s} +1 = w_{m-s+1}$. Likewise $\dim D_r \cap U$ increases .when  $w_{m-s} = w_{m-s+1}$. Therefore the set $\{ j \in [m] \ | \ n_{j+1} = n_j +1 \} =  \{ m+1 - i \ | \ i \not \in \alpha \}$. Thus $\tau(\Omega_\alpha^{\mathcal{A}}) =  \Omega_\beta^{\tau(\mathcal{A})}$. \end{proof}

For a contravariant mapping $\tau \in Aut(\G_{\ell, m})$ we find when $\tau(\Omega_\alpha^{\mathcal{A}}) = \Omega_{\alpha}^{\mathcal{A}}$.

\begin{lem}

Let $\tau \in Aut(\G_{\ell, m})$ be a contravariant mapping, $ \alpha = (a_1, a_2, \ldots, a_{\ell} )\subseteq [m]$ and $\mathcal{A}$ an $\alpha$--flag. Then  $\tau(\Omega_\alpha^{\mathcal{A}}) = \Omega_{\alpha}^{\mathcal{A}}$.
 if and only if $\alpha = \{ m+1-j \ | \ j \not\in \alpha \}$ and the sets $\{ \tau(A_i) \ | \ a_i \in \alpha_{nc} \} = \{ A_{i} \ | \ a_i \in \alpha_{nc}\}$ are equal.
\label{lem:Contravariant}
\end{lem}
\begin{proof}  Lemma \ref{lem:dualSchubertaction} and Proposition \ref{prop:KL} imply that $\tau(\Omega_\alpha^{\mathcal{A}}) = \Omega_{\alpha}^{\tau(\mathcal{A})}$. Theorem \ref{thm:SchubertEqualityTheorem} states that $\Omega_{\alpha}^{\tau(\mathcal{A})} = \Omega_{\alpha}^{\mathcal{A}}$ if and only if they have the same subspaces in for the nonconsecutive indices. \end{proof}

\begin{thm} An automorphism of $\G_{\ell, m}$ is an automorphism of $\Omega_{\alpha}^{\mathcal{A}}$ if and only if it  maps the set $\{  A_i \ | \ a_i \in \alpha_{nc}\}$ to itself. \end{thm} \begin{proof} It follows from Teorem \ref{thm:Variant} and Lemma \ref{lem:Contravariant}.\end{proof}


\end{document}